\documentclass[letterpaper, 10pt,conference]{ieeeconf}

\IEEEoverridecommandlockouts 
\usepackage{amsmath,amssymb,amsfonts}
\usepackage{graphicx}
\usepackage{textcomp}
\usepackage{graphicx}      

\usepackage{algorithm}
\usepackage{algorithmicx}
\usepackage{color}
\usepackage{bm}            
\usepackage{amsfonts}
\usepackage{mathtools}      
\usepackage{mathrsfs}

\usepackage{url}
\usepackage{hyperref}  
\usepackage{lettrine}
\usepackage{comment}

\newcommand{\p}{\bm{\mu}}      
\newcommand{\pdim}{n_{\mu}}         
\newcommand{\nx}{n_x}      
\newcommand{\ny}{n_y}
\newcommand{\nn}{n_n}
\newcommand{\udim}{n_u}

\newcommand{\msig}{\p_{\sigma}}
\newcommand{\psig}{p_{\sigma}}

\newcommand{\sigSet}{\sigma \in \Sigma}
\newcommand{\word}{\sigma_{1}\sigma_{2}\cdots \sigma_{k}}
\newcommand{\wordSet}[1]{w \in \Sigma^{#1}}
\newcommand{\expect}[1]{E\left[#1 \right]}
\newcommand{\filt}{\mathcal{F}_{t}^{\p,-}}
\newcommand{\filtp}{\mathcal{F}_{t}^{\p,+}}

\newcommand{\filtr}{\mathcal{F}_{t}^{\mathbf{r}}}

\newcommand{\uw}[1]{\p_{#1}}

\newcommand{\zwr}{\mathbf{z}^{\mathbf{r}}_{w} }
\newcommand{\zvr}{\mathbf{z}^{\mathbf{r}}_{v} }
\newcommand{\zsvr}{\mathbf{z}^{\mathbf{r}}_{\sigma v} }

\newcommand{\timeset}{t \in \mathbb{Z}}
\newcommand{\yb}{\mathbf{y}}

\newcommand{\vb}{\mathbf{v}}
\newcommand{\eb}{\mathbf{e}}
\newcommand{\xb}{\mathbf{x}}
\newcommand{\ub}{\mathbf{u}}
\newcommand{\z}{\mathbf{z}}
\renewcommand{\r}{\mathbf{r}}

\newtheorem{Notation}{Notation}
\newtheorem{Definition}{Definition}
\newtheorem{Example}{Example}
\newtheorem{Theorem}{Theorem}
\newtheorem{Corollary}{Corollary}
\newtheorem{Lemma}{Lemma}
\newtheorem{Remark}{Remark}
\newtheorem{Assumption}{Assumption}
\newenvironment{pf}[1]{\begin{proof}{#1}}{\end{proof}}

\let\labelindent\relax
\usepackage{enumitem}
\renewcommand{\theenumi}{\arabic{enumi}}

\renewcommand{\theenumii}{\arabic{enumii}}

\renewcommand{\theenumiii}{\arabic{enumiii}}

\renewcommand{\theenumiv}{\arabic{enumiv}}

\newcommand{\btheta}{\boldsymbol{\theta}}

\newcommand{\GBS}{\textbf{GBS}}

\newcommand{\y}{\textbf{y}}
\newcommand{\bu}{\textbf{u}}
\newcommand{\x}{\textbf{x}}

\newcommand{\e}{\textbf{e}}

\makeatletter
\newcommand{\mathleft}{\@fleqntrue\@mathmargin0pt}
\makeatother

\title{ Towards stochastic realization theory for Generalized Linear Switched Systems with inputs: 
       decomposition into stochastic and deterministic components and existence and uniqueness of innovation form}
\author{Elie Rouphael, Manas Mejari, Mihaly Petreczky,  Lotfi Belkoura%
\thanks{E. Rouphael, M. Petreczky and  L. Belkoura are with
Univ. Lille, CNRS, Centrale Lille, UMR 9189 CRIStAL, Lille, France.
{\tt\small first.lastname@univ-lille.fr}, 
}
\thanks{M. Mejari is with Swiss AI lab IDSIA-SUPSI, Lugano, Switzerland.
{\tt\small manas.mejari@supsi.ch}%
}
}

\begin{document}
\maketitle
\begin{abstract}
In this paper, we study a class of stochastic Generalized Linear Switched System (GLSS), which includes subclasses of jump-Markov, piecewide-linear and Linear Parameter-Varying (LPV) systems. 
We prove that the output of such systems can be decomposed into deterministic and stochastic components. Using this decomposition, we show existence of state-space representation in innovation  form, and we provide sufficient conditions for such representations to be  minimal and unique up to isomorphism. 
\end{abstract}

\section{Introduction}
\label{sect:intro}
A discrete-time stochastic \emph{Generalized Linear Switched System (abbreviated as GLSS)} is a system of the form 
\begin{equation}
\label{eqn:LPV_SSA}
	\mathbf{S} \left\{ 
        \begin{aligned}
        & \mathbf{x}(t+1) = \sum_{i=1}^{\pdim} (A_i\xb(t)+B_i\ub(t)+K_i\vb(t))\p_i(t) \\
	& \mathbf{y}(t) = C\mathbf{x}(t) + D\mathbf{u}(t) + F \vb(t), ~~ t \in \mathbb{Z}
       \end{aligned} \right.
\end{equation}
where
$A_{i} \in \mathbb{R}^{\nx \times \nx } $, $B_{i} \in \mathbb{R}^{\nx \times \udim } $, $K_{i} \in \mathbb{R}^{\nx \times \nn }$, $i=1, \ldots, \pdim$,  $C \in \mathbb{R}^{\ny \times \nx}$ and $D \in \mathbb{R}^{\ny \times \udim}$, $F \in \mathbb{R}^{\ny \times \nn}$  are constant matrices. 
The stochastic process $\xb,\ub,\yb,\vb$ and $\p=(\p_1,\ldots,\p_{\pdim})^T$
are the state, input, output, noise and switching processes respectively,
taking values respectively in
$\mathbb{R}^{\nx}, \mathbb{R}^{\udim},\mathbb{R}^{\ny},
\mathbb{R}^{\nn}, \mathbb{R}^{\pdim}$ and defined on $\mathbb{Z}$.

Intuitively, \eqref{eqn:LPV_SSA} is a generalization of linear switched systems to the case of infinitely many discrete modes.  
If $\p(t)$ takes values in the set of unit vectors, i.e.,
$\p_i(t) \in \{0,1\}$, $i=1,\ldots,\pdim$, $\sum_{i=1}^{\pdim} \p_i(t)=1$, then \eqref{eqn:LPV_SSA} is a \emph{switched system \cite{Sun:Book}.} 
If, in addition,
the process $\theta(t)$, defined by
$\theta(t)=i\!\! \iff \!\! \p_i(t)=1$, is a Markov chain,
then \eqref{eqn:LPV_SSA} is a jump-Markov system 
\cite{CostaBook}.
If $\p$ takes values from a possibly infinite set and 
$\p_1=1$, then  \eqref{eqn:LPV_SSA} could be viewed as a subclass of 
\emph{linear parameter varying (LPV)} systems \cite{Toth2010SpringerBook}
	with an \emph{affine dependence}  on scheduling, and $(\p_2,\ldots,\p_{\pdim})^T$ 
 corresponds to the scheduling 
signal. However, in contrast to LPV systems, in general we are
agnostic about  the role of $\p$ in control, hence the use of the term GLSS. 

\textbf{Context and motivation:}
 Consider the  following deterministic counterpart of \eqref{eqn:LPV_SSA}
\begin{equation}
\label{eqn:LPV_SSAd}
	\begin{split}
        & x(t+1) = \sum_{i=1}^{\pdim} (A_ix(t)+B_iu(t)+K_iv(t))\mu_i(t) \\
	& y(t) = Cx(t) + Du(t) + F v(t), t \in \mathbb{Z}
       \end{split}
\end{equation}
where all the signals are deterministic
and which 
      describes the response of the true system for \textbf{any} input, switching and noise realization, and not only for the samples from $\bu,\p,\vb$.
For the tuple of matrices 
$S=(\{A_i,B_i,K_i\}_{i=0}^{\pdim},C,D,F)$
of \eqref{eqn:LPV_SSA} define the \emph{deterministic behavior}  $\mathcal{B}_S$
of $S$ as the as the set of all 
tuples of deterministic trajectories $(u,\mu,y)$  such that there exists trajectories $x$ and $v$ for which \eqref{eqn:LPV_SSAd} 
holds. Clearly, all samples paths of $(\bu,\p,\vb)$ belong to $\mathcal{B}_{S}$.

In system identification of switched and LPV systems, we want to find matrices 
	$\hat{S}=(\{\hat{A}_i,\hat{B}_i,\hat{K}_i\}_{i=0}^{\pdim},\hat{C},\hat{D},\hat{F})$ from one single tuple of trajectories $(u,\mu,y)$ from
$\mathcal{B}_S$, such that the deterministic behavior $\mathcal{B}_{\hat{S}}$ is an approximation of 
$\mathcal{B}_{S}$. For jump-Markov system, the task is similar, 
but in the definition of the deterministic 
behavior 
the switching and noise trajectories are sampled from processes $\p$ and $\vb$.

	As a tool to achieve this goal, we assume that
   the data $(u,\mu,y)$ used for identification,
    are sampled from the processes
    $(\bu,\p,\y)$, where $\y$ is 
    generated by the data generator
    \eqref{eqn:LPV_SSA}.
	Pragmatically, we cannot exclude random effects during the identification experiment (measurement error, etc.), and
    randomness can be used for  statistical reasoning about algorithms.
    For the latter, we view $\hat{S}$ as a statistics for 
    the matrices of \eqref{eqn:LPV_SSA}, which, in turn,
    parameterize the joint distribution of $(\bu,\p,\y)$.
    However,  the estimated models should
    approximate the true system for
    all inputs and scheduling signals, and
    not only those which are sampled from
    $\bu$ and $\p$.
    %
%
Unfortunately, good statistical properties 
of $\hat{S}$, e.g., consistency, guarantee only that the output of the stochastic system determined by $\hat{S}$ is \emph{close} to that of the stochastic data generator \eqref{eqn:LPV_SSA}, 
for the  specific stochastic input $\bu$ and switching $\p$ and a choice of the noise process.
However, this does not generally imply that the deterministic behaviors $\mathcal{B}_{\hat{S}}$ and $\mathcal{B}_{S}$ of $\hat{S}$ and $S$ are close.
In fact, the outputs of two GLSS may be the same for the input 
$\bu$ and switching $\p$ used during identification,  but  be different for others \cite{rouphael2022minimal}.

In the LTI case, this was resolved by assuming that the data generator \eqref{eqn:LPV_SSA} and  the stochastic system corresponding to  $\hat{S}$ 
are minimal and in \emph{innovation form}. 
Since the matrices of  minimal systems in innovation form  
with the same outputs and inputs are related by a basis transformation  \cite{LindquistBook},
if the stochastic system corresponding to $\hat{S}$ is close to the data generator, then intuitively the matrices
$\hat{S}$ and $S$ are close after a suitable basis transformation. Hence, the deterministic behaviors of $\hat{S}$ and $S$
are close.
Moreover,  innovation form  was useful for developing system identification algorithms (e.g., subspace, prediction error minimization), and establishing a correspondence  between 
state-space representations and optimal predictors of $\y$ based on its past of $\bu$ and $\y$. 

A key step in the proof of existence and uniqueness of minimal LTI systems in innovation form is the 
decomposition $\y't)=\y^d(t)+\y^s(t)$ of the output
into two components. 
The component $\y^d$ is determined by the input and it is the output of a noiseless LTI system.  The component $\y^s$ is the output of an autonomous LTI system and it depends only on the noise. 
Then existence and uniqueness of minimal LTI systems in innovation form  follow by
applying realization theory
\cite{LindquistBook} to the autonomous LTI system generating $\y^s$.

\textbf{Contribution:}
In this paper we show an analogous result for GLSSs, i.e. we show that  $\y(t)=\y^d(t)+\y^s(t)$, where $\y^d$ is the output of a GLSS with no noise $\vb$, and $\y^s$ is the output of a GLSS with no input $\ub$. 
Furthermore, by using results on realization theory of 
GLSSs with no inputs \cite{PetreczkyBilinear}, 
we use this decomposition 
to show  existence of an innovation form for GLSSs with inputs. Moreover, we present sufficient conditions for
minimality and uniqueness  (up to change of basis) of 
GLSSs in innovation form. 
This  eans that minimal GLSSs in innovation form have the same
useful properties for system identification as their LTI counterparts. In particular, if two minimal GLSSs in innovation form generate (approximately) the same output for the data used of identification, then they will generate  (approximately) the same
output for any input and switching signal, including those which do not satisfy the assumptions of the identification experiment. 

\textbf{Related work:}
System identification in general, and subspace methods in particular,  of switched \cite{LauerBook}, jump-Markov \cite{JumpMarkov1} and LPV systems \cite{Toth2010SpringerBook,CoxTothSubspace,Wingerden09,pigalpv,Verdult02,Verdult04,tanaka2023state} is a well-established topic.
Stochastic realization theory of GLSSs with no inputs, i.e., jump-Markov systems with no inputs, bilinear systems  and autonomous stochastic LPV systems 
were addressed in \cite{PetreczkyBilinear}. With respect to \cite{PetreczkyBilinear}
the main difference is the presence of the control input $\bu$.

Existence of a decomposition and existence of innovation form appeared in
\cite{MejariLPVS2019}, but only for the case of 
LPV systems with zero mean i.i.d. scheduling,.
With respect to \cite{MejariLPVS2019}, the main
novelty is that we allow more general switching processes, including finite Markov chains, and that we address minimality and uniqueness 
of innovation representations. 
Moreover, 
in \cite{MejariLPVS2019} the proofs were  not presented, they were included in the report \cite{mejari2019realization}, 
the latter can be viewed as a preliminary version of this paper. 

The existence of innovation representation  was studied for LPV systems in \cite{CoxLPVSS,CoxTothSubspace}.
In contrast to this paper, in \cite{CoxLPVSS,CoxTothSubspace} the noise gain of the innovation representation has a dynamical dependence on the scheduling, and there is no claim on minimality and uniqueness of such representations.
In particular, it is unclear when the innovation representation from \cite{CoxLPVSS,CoxTothSubspace} 
generates the same output as the original model for all scheduling signals.
However, \cite{CoxLPVSS,CoxTothSubspace}  has the advantage that it is valid for any scheduling signal. 

\section{Preliminaries}
\label{sect:prelim}
Below we present the notation used in the paper. In addition, we recall a number of concepts from \cite{PetreczkyBilinear}, and then we use them to define the subclass of stationary GLSS for which our main results hold. 

\textbf{Probability theory:}
We use the standard terminology of probability theory \cite{Bilingsley}. All the random variables and stochastic processes are understood w.r.t. to a fixed probability space $\left(\Omega, \mathcal{F}, \mathcal{P}\right)$, where $\mathcal{F}$ is a $\sigma$-algebra over the sample space $\Omega$.
The expected value of a random variable $\mathbf{r}$ is denoted by $E[\mathbf{x}]$ and conditional expectation w.r.t. $\sigma$-algebra $\mathcal{F}$ is denoted by $\expect{\mathbf{r} \mid \mathcal{F}}$. 
The stochastic processes in this paper are discrete-time ones defined over the time-axis $\mathbb{Z}$.

\textbf{Finite sequences:}
In what follows  $\Sigma=\{1,\ldots,\pdim\}$.
A \emph{non empty word} over $\Sigma$ is a finite sequence of letters (elements) of $\Sigma$, i.e.,  
$w = \sigma_{1}\sigma_{2}\cdots \sigma_{k}$, for some $k \in \mathbb{N}$, $k >0$, $\sigma_{1}, \sigma_{2}, \ldots, \sigma_{k} \in \Sigma$; $|w|:=k$ is the 
length of $w$. The set of \emph{all} nonempty words is denoted by $\Sigma^{+}$. We denote the \emph{empty word} by $\epsilon$ and 
by convention $|\epsilon|=0$. Let $\Sigma^{*} = \{\epsilon \} \cup \Sigma^{+}$. The concatenation of two nonempty words $v = a_{1}a_{2}\cdots a_{m}$ and  $w= b_{1}b_{2}\cdots b_{n}$ is defined as $vw = a_{1}\cdots a_{m} b_{1} \cdots b_{n}$ for some $m,n > 0$. By convention $v\epsilon = \epsilon v = v$ for all $v \in \Sigma^{*}$. 

\textbf{Notation for matrices}
    We denote by $I_n$ the $n \times n$
    identity matrix. 
	Consider $n \times n$ square matrices $\{A_{\sigma}\}_{\sigSet}$. For any 
 $w = \sigma_{1}\sigma_{2}\cdots \sigma_{k} \in \Sigma^{+}$, $k\!>\!0$ and $\sigma_{1}, \ldots, \sigma_{k} \in \Sigma$, we define 
	$A_{w} = A_{\sigma_{k}}A_{\sigma_{k-1}}\cdots A_{\sigma_{1}}$.
	For an empty word $\epsilon$,  set $A_{\epsilon} = I_n$.

\textbf{Notions from \cite{PetreczkyBilinear}:
admissible switching, ZMWSII, SII processes:}
We first formulate our assumptions for the switching process.
For every word $\wordSet{+}$ where $w=\word$, $k \geq 1$, $\sigma_{1},\ldots, \sigma_{k} \in \Sigma$, we define the  process $\uw{w}$ as follows:
\begin{equation}
\label{eqn:uw}
		\uw{w}(t) = \uw{\sigma_{1}}(t-k+1)\uw{\sigma_{2}}(t-k+2)\cdots\uw{\sigma_{k}}(t)  
\end{equation}
For an empty word $w= \epsilon$, we set $\uw{\epsilon}(t)=1$. 
\begin{Definition}[Admissible process, \cite{PetreczkyBilinear}]
	\label{asm:A1}	
	The switching process $\p$ is called \emph{admissible}, if the following holds:

\textbf{1.} There exists a set $\mathcal{E} \subseteq \Sigma \times \Sigma$ such that:
\begin{itemize}
    \item[\textbf{--}] $\forall \sigma \in \Sigma, \exists \sigma' \in \Sigma : (\sigma,\sigma') \in \mathcal{E}$.
    \item[\textbf{--}] Define the set of admissible words $L$ as the set of all words $w \in \Sigma^{+}$ such that $w = \sigma_1\cdots\sigma_{k} \in \Sigma^{+}, \sigma_1,\dots,\sigma_{k} \in \Sigma, k > 0$ and for all $i = 1,\dots,k-1$, $(\sigma_i,\sigma_{i+1}) \in \mathcal{E}$. Then for all $w \in \Sigma^{+},w \notin L, \p_w=0$.
\end{itemize} 
 
\textbf{2.}
  Denote by $\mathscr{F}^{\p,-}_{t}$ the $\sigma$-algebra generated by the random variables $\{\p(k) \}_{k < t}$. There exists positive numbers
  $\{p_{\sigma}\}_{\sigma \in \Sigma}$ such that for any $w,v \in \Sigma^{+}, \sigma, \sigma^{'} \in \Sigma$, $t \in \mathbb{Z}$:
  \mathleft \begin{align*}
      & E[\uw{w\sigma}(t)\uw{v\sigma^{'}}(t) \mid \mathscr{F}^{\p,-}_{t} ]= \\ 
      &\left\{\begin{array}{rl}
           \!\!  p_{\sigma} \uw{w}(t\!-\!1)\uw{v}(t\!-\!1) & \sigma=\sigma^{'} \text{and} \ w\sigma,v\sigma \in L \\
        0 & \mbox{otherwise} 
        \end{array}\right. \\
     &  E[\uw{w\sigma}(t)\uw{\sigma^{'}}(t) \mid \mathscr{F}^{\p,-}_{t}] = \\
     &\left\{\begin{array}{rl}
      p_{\sigma} \uw{w}(t-1)  & \sigma=\sigma^{'} \mbox{ and } \ w\sigma \in L \\
    0 & \mbox{otherwise}
    \end{array}\right. 
    \end{align*}
    \begin{align*}
    & E[\uw{\sigma}(t)\uw{\sigma^{'}}(t) \mid \mathscr{F}^{\p,-}_{t}] = \begin{array}{rl}
    0 & \mbox{if } \sigma \ne \sigma^{'} 
    \end{array}
  \end{align*}
  
\textbf{3.}
 There exist real numbers $\{\alpha_{\sigma}\}_{\sigma \in \Sigma}$ such that 
 $\sum_{\sigma \in \Sigma} \alpha_{\sigma} \uw{\sigma}(t)=1$ for all $t \in \mathbb{Z}$. 
 
\textbf{4.}
  For each $w,v \in \Sigma^{+}$, the process $\begin{bmatrix} \uw{w}, \uw{v} \end{bmatrix}^T$ is wide-sense stationary.
\end{Definition}

Below we recall from \cite{PetreczkyBilinear} a number of examples of admissible processes.
 \begin{Example}[White noise]
\label{i.i.d. input}
 If $\p = [\p_1, \p_{2}, \ldots, \p_{\pdim}]^{T}$ is 
 an i.i.d. process such that $\p_1=1$,
 for all $i,j=2,\ldots, \pdim$, $t \in \mathbb{Z}$, $\p_i(t), \p_j(t)$ are independent and
$\p_i(t)$ is zero mean, then $\p$ is admissible with $\mathcal{E}=\Sigma \times \Sigma$ and $p_{\sigma}=E[\p_{\sigma}^2(t)]$. 
 \end{Example}
 \begin{Example}[Discrete valued i.i.d process]
 \label{disc:input}
  Let $\btheta$ be an i.i.d process 
  with values in $\Sigma=\{1,\ldots,\pdim\}$.  Let $\p_{\sigma}(t)=\chi(\btheta(t)=\sigma)$ for all $\sigma \in \Sigma$, $t \in \mathbb{Z}$, where $\chi$ is the indicator function. Let $\mathcal{E}=\Sigma \times \Sigma$, and  
$p_{\sigma}=P(\btheta(t)=\sigma)$, $\alpha_{\sigma}=1$ for all $\sigma \in \Sigma$. 
  Then 
	  $\p$ is admissible.
 \end{Example}
  \begin{Example}[Markov chain]
 \label{markov:process}
  Assume that $\btheta$  is a stationary and ergodic Markov process whose state space is the finite set $\Theta$. Assume
$P(\btheta(t)=q_2 \mid \btheta(t-1)=q_1) = p_{(q_2,q_1)} > 0$, $q_1,q_2 \in \Theta$. 
 Let us take $\Sigma=\Theta \times \Theta$, $\p_{(q_2,q_1)}(t)=\chi(\btheta(t+1)=q_2, \btheta(t)=q_1)$ for all $q_1,q_2 \in \Theta, t \in \mathbb{Z}$, and
 let 
 $\mathcal{E}=\{ (\sigma_1,\sigma_2) \in \Sigma \times \Sigma \mid \sigma_1=(q_2,q_1),\sigma_2=(q_3,q_2), q_1,q_2,q_3 \in \Theta \}$.
	  Define $\alpha_{\sigma}=1$ for all $\sigma \in \Sigma$.  Let us identify $\Sigma$ with the set $\{1,\ldots, \pdim\}$,
	  where $\pdim$ is the square of cardinality of $\Theta$.
	  Then $\p$ is admissible.
 \end{Example}
 \begin{Assumption}
  The switching process $\p$
 is admissible. 
 \end{Assumption}
 This assumption imposes restrictions on
 the data used for system identification, but not necessarily for the model class which will be identified.
 It can be thought of as a persistence of excitation
 condition. In particular, 
 binary and white noises, which are the simplest persistently exciting signals, satisfy our assumption.  
We believe that for developing realization theory for general switching signals,
these  simple cases must be understood first. Moreover, admissible switching signals  cover the fairly general case of Markov chains. 
 \begin{Remark}[LPV systems]
  For LPV systems, our assumption
  implies that the scheduling signal 
  used for 
  identification is sampled from a stochastic process. In addition to this being a persistence of excitation condition, we argue
   that in  certain cases this 
   assumption is justified
   \cite{pigalpv}: 
   in the presence of measurement errors, or when the scheduling is externally generated,
  or it is a function of the stochastic states/inputs.
\end{Remark}
 

Next, we recall the concept of \emph{ZMWSII process w.r.t $\p$} from \cite{PetreczkyBilinear}.
To this end, let $\{p_{\sigma}\}_{\sigma \in \Sigma}$  be the constants from Definition \ref{asm:A1}. 
%
For any $w=\sigma_1\cdots \sigma_k \in \Sigma^{+}$,
$\sigma_1,\ldots, \sigma_k \in \Sigma$, 
for a process $\mathbf{r}$, 
define the product $p_w$ and the  process $\zwr$
\begin{equation}\label{eqn:zwu}
\begin{aligned}
    p_w&=p_{\sigma_1}p_{\sigma_2} \cdots p_{\sigma_k}, \\
	\zwr(t) &= \mathbf{r}(t-|w|) \uw{w}(t-1)\frac{1}{\sqrt{p_{w}}},
\end{aligned}
\end{equation}
where $\uw{w}$ is as in \eqref{eqn:uw}.
The process $\zwr$ in \eqref{eqn:zwu} is interpreted as the product of the \emph{past} of $\r$ and $\p$.
\begin{Definition}[ZMWSSI, \cite{PetreczkyBilinear}]\label{def:ZMWSSI}
	A process $\r$ is \emph{Zero Mean Wide Sense Stationary w.r.t. $\p$} (ZMWSSI) if 

	\textbf{(1)} For $\timeset$, the $\sigma$-algebras generated by the variables $\{\mathbf{r}(k) \}_{k \leq t}$, $\{\msig(k) \}_{k < t, \sigSet}$ and $\{\msig(k) \}_{k \geq t, \sigSet},$ denoted by $\filtr$, $\filt$ and $\filtp$ respectively, are such that $\filtr$ and $\filtp$ are conditionally independent w.r.t. $\filt$.

	 \textbf{(2)}
			The processes $\{\mathbf{r}, \{\zwr \}_{\wordSet{+}} \}$  are zero mean, square integrable and are jointly wide sense stationary.
\end{Definition} 

Intuitively, ZMWSII is an extension of wide-sense stationarity, where 
$\Sigma^{+}$ is viewed as time axis:
ZMWSII implies 
the covariances $\expect{\zwr(t)(\zvr(t))^{T}}$ do not depend on $t$, and they depend on the difference between $v$ and $w$.
Next, we recall the definition of a square integrable process w.r.t. $\p$. 
\begin{Definition}[SII process, \cite{PetreczkyBilinear}]
	A process $\r$ is \emph{Square Integrable w.r.t. {$\p$} (SII)}, if for all $\wordSet{*}$, $\timeset$,
	the random variable $\mathbf{r}(t+|w|) \p_{w}(t+|w|-1)$
	is square integrable.
\end{Definition}
As it was mentioned in \cite[Section III, Remark 2]{PetreczkyBilinear}, if $\p$ is essentially bounded, then any ZMWSII process is SII.

\textbf{Assumptions, inputs and outputs and on GLSSs:}
First, we define the notion of white noise processes w.r.t. $\p$, which will be used for stating our assumptions on GLSSs.
\begin{Definition}[White noise w.r.t. $\p$]
	A ZMWSII process $\mathbf{r}$ is a white noise w.r.t. $\p$, if for all $w \in \Sigma^{+}$, $v \in \Sigma^{*}$, 
     $\sigma \in \Sigma$,
	\[ E[\zwr(t)(\zsvr(t))^T] \!\!=\!\! \left\{\begin{array}{ll}
        0 & \mbox {if } w \ne \sigma v \\
        E[\z^{\r}_{\sigma}(t)(\z^{\r}_{\sigma}(t))^T]  & \mbox{if } w=\sigma v 
    \end{array}\right., 
     \]
      and $E[\z^{\r}_{\sigma}(t)(\z^{\r}_{\sigma}(t))^T]$ is
      nonsingular for all $\sigma \in \Sigma$.
\end{Definition}
Intuitively,  if $\r$ is a white noise process w.r.t. $\p$, then
$\{\z_{w}^{\r}(t)\}_{w \in \Sigma^{+}}$ is a sequence of
uncorrelated random variables. Due to \textbf{3.} of Definition \ref{asm:A1}, the product $\r(t-k)\r(t)$
is a linear combination of $\{\z_{w}^{\r}(t)\}_{w \in \Sigma^{+}}$, hence a white noise process w.r.t. $\p$ is also
a white noise process in the classical sense. 
Conversely, if $\mathbf{r}$ is a white noise and it is independent of
$\{\p(s)\}_{s \in \mathbb{Z}}$, then it is
a white noise process w.r.t. $\p$. 

\begin{Assumption}[Inputs and outputs]
	\label{asm:main}\textbf{(1)}
	$\mathbf{u}$ is a white noise  w.r.t. $\p$, 
	\textbf{(2)} the process
	$\begin{bmatrix} \mathbf{y}^T, \!\! & \!\! \mathbf{u}^T \end{bmatrix}^T$ is a ZMWSSI. 

\end{Assumption}

The assumption that $\bu$ is white noise was made for the sake of simplicity, we conjecture that the results of the paper can be extended to more general inputs, e.g., inputs generated by autoregressive models driven by white noise.
Next,  we define the class of systems considered in this paper.
\begin{Definition}[Stationary GLSS]
\label{defn:LPV_SSA_wo_u}
\label{def:Stationary}
  A  \emph{stationary GLSS} (abbreviated 
  sGLSS) of $(\y,\bu,\p)$ is a system \eqref{eqn:LPV_SSA}, such that
\begin{itemize}

	\item[\textbf{1.}]
		$\mathbf{w}=\begin{bmatrix} \vb^T, \!\! & \!\! \bu^T \end{bmatrix}^T$ is a white noise process w.r.t. $\p$. 

    \item[\textbf{2.}]
  The process
  $\begin{bmatrix} \xb^T\!\!, & \!\! \mathbf{w}^T \end{bmatrix}^T$  is a ZMWSSI, and 
     $    
        E[\z^{\xb}_{\sigma}(t)(\z^{\mathbf{w}}_{\sigma}(t))^T]=0, ~
        E[\xb(t)(\z^{\mathbf{w}}_w(t))^T]=0,
    $
  for all $\sigSet$, $w \in \Sigma^{+}$.

	
	\item[\textbf{3.}]
		The eigenvalues of the matrix $\sum_{\sigSet} \psig {A}_{\sigma} \otimes {A}_{\sigma}$ are inside the open unit circle.

    \item[\textbf{4.}]  For all $\sigma_1,\sigma_2 \in \Sigma$, if $(\sigma_1,\sigma_2) \notin \mathcal{E}$, then
  $A_{\sigma_2}A_{\sigma_1}=0$ and $A_{\sigma_2}\begin{bmatrix} B_{\sigma_1} & K_{\sigma_1} \end{bmatrix}  E[\z^{\mathbf{w}}_{\sigma_1}(t) (\z_{\sigma_1}^{\mathbf{w}}(t))^T]=0$.

       \end{itemize}
If $B_i=0$, $i \in \Sigma$, and $D=0$  the we call \eqref{eqn:LPV_SSA} an \emph{autonomous stationary} GLSS (asGLSS) of $(\y,\p)$.
\end{Definition}
Intuitively, sGLSSs are introduced in order to ensure that all the relevant stochastic processes are stationary in an suitable sense. The latter assumption is  widespread in stochastic realization theory and system identification.  

In the terminology of \cite{PetreczkyBilinear}, a sGLSS (resp. asGLSS)  corresponds to a stationary \emph{Generalized Bilinear System} (\GBS)\  
with noise $\begin{bmatrix} \vb^T & \bu^T \end{bmatrix}^T$ (resp. $\vb$).
From \cite{PetreczkyBilinear} it follows that the state and output process $\xb$ and $\yb$ are ZMWSII, and hence
Assumption \ref{asm:main} is satisfied for all $(\bu,\p,\y)$ generated by sGLSS.
Moreover, 
from \cite[Lemma 2]{PetreczkyBilinear} it follows that
\begin{equation*}
	\label{stat:state:eq1}
	\xb(t)= \sum_{\substack{\sigma \in \Sigma, w \in \Sigma^{*},  \\ \sigma w \in L} } \sqrt{p_{\sigma w}} A_w\Bigg (K_{\sigma}\z^{\vb}_{\sigma w}(t) + B_{\sigma} \z^{\bu}_{\sigma w}(t) \Bigg)
\end{equation*}
where the infinite sum converges in the mean square sense.
Hence, the state $\xb$ is uniquely determined by the system matrices and the input $\ub$ and noise $\vb$, and it is the limit of
any state trajectory started from some initial state. 
\begin{Notation}
      We identify the sGLSS $\mathbf{S}$ of the form \eqref{eqn:LPV_SSA} with the tuple 
	$\mathbf{S} =(\{A_{\sigma},K_{\sigma},B_{\sigma}\}_{\sigma=1}^{\pdim},C,D,F,\vb)$, and if
    $\mathbf{S}$ is a asGLSS,
    i.e. $B_{\sigma}=0$, $\sigSet$, $D=0$, 
    then we will identify it with the tuple
	$\mathbf{S} =(\{A_{\sigma},K_{\sigma}\}_{\sigma=1}^{\pdim},C,F,\vb)$.
\end{Notation}

\section{Main result}
\label{sect:decomp}
We start by recalling from \cite{PetreczkyBilinear} the following terminology.
\begin{Notation}[Orthogonal projection $E_l$]
	\label{hilbert:notation}
	Recall from \cite{Bilingsley} that
    the set of real valued square integrable
	random variables, denoted by $\mathcal{H}_1$, 
    forms a Hilbert-space with the scalar product defined as $\langle \mathbf{z}_1,\mathbf{z}_2\rangle=E[\mathbf{z}_1\mathbf{z}_2]$. 
	Let $\mathbf{z}$ be a square integrable 
	random variable taking its values in $\mathbb{R}^k$.  Let $M$ be a closed subspace   of $\mathcal{H}_1$. 
	The  orthogonal projection of $\mathbf{z}$ onto $M$, \emph{denoted by $E_l[\mathbf{z} \mid M]$},
	is defined as the 
     random variable $\mathbf{z}^{*}=\begin{bmatrix} \mathbf{z}_1^{*},\ldots,\mathbf{z}_k^{*} \end{bmatrix}^T$ such that $\mathbf{z}_i^{*} \in M$ is the orthogonal projection of the $i$th coordinate $\mathbf{z}_i$ of $\mathbf{z}$, viewed as an element of
     $\mathcal{H}_1$ onto $M$. 
	If $\mathfrak{S}$ is a subset of square integrable random variables in $\mathbb{R}^p$,  and 
	   $M$ is generated by the coordinates of the elements of $\mathfrak{S}$,
    then instead of $E_l[z \mid M]$ we use \( E_{l}[\mathbf{z} \mid \mathfrak{S}] \). 
\end{Notation}
Intuitively, 
$E_l[\mathbf{z} \mid \mathfrak{S}]$ is \emph{the best 
(minimal variance) linear prediction} of $\mathbf{z}$ based on the elements of $\mathfrak{S}$.

Using the notation above, let us define the \emph{deterministic component} $\yb^d$ of $\yb$ as follows
	\begin{equation}
		\label{decomp:outp:eq1}
		{\yb}^d(t)=E_l[\yb(t) \mid \{\z_w^{\ub}(t)\}_{w \in \Sigma^{+}} \cup \{\ub(t)\}]. 
	\end{equation}
	Also, define the \emph{stochastic component} of $\yb$ as 
	\begin{equation}
		\label{decomp:outp:eq2}
		\yb^s(t)=\yb(t)-\yb^d(t).
	\end{equation}
Intuitively, $\y^d$ represent the best  prediction of $\y$ which is
linear in the present and past  values of $\bu$ and non-linear in the past values of $\p$.
%
In fact, 
 $\y^d$ is the output of
the asGLSS obtained from 
\eqref{eqn:LPV_SSA} by considering $\vb=0$ and
viewing $\bu$ as noise, and 
$\y^s$ is the output of
the asGLSS obtained from
\eqref{eqn:LPV_SSA} by taking $\bu=0$
and viewing $\vb$ as noise.
\begin{Theorem} 
	\label{decomp:lemma}
	  For a  sGLSS 
     of the form \eqref{eqn:LPV_SSA}, 
	$\mathbf{S}_d=(\{A_{\sigma},B_{\sigma}\}_{\sigma=1}^{\pdim},C,D,\bu)$ is an asGLSS of $(\y^d,\p)$ and 
	$\mathbf{S}_s=(\{A_{\sigma},K_{\sigma}\}_{\sigma=1}^{\pdim},C,F,\vb)$ is an asGLSS of $(\y^s,\p)$.
\end{Theorem}
The proof of Theorem \ref{decomp:lemma} is presented in the Appendix \ref{appendix1}.


In fact, the converse of Theorem \ref{decomp:lemma} also holds.  To this end, recall from \cite{PetreczkyBilinear} the definition of   \emph{innovation process} of $\yb^s$: 
	\begin{equation} 
		\label{decomp:lemma:innov}
		\eb^s(t)=\yb^s(t)-E_l[\yb^s(t) \mid\{\z^{\yb^s}_w (t)\}_{w \in \Sigma^{+}}]
	\end{equation}
 Intuitively, $\eb^s(t)$ is the difference between $\y(t)$ and the best linear prediction of $\y^s(t)$ based on its own past multiplied with past values of the switching process.
\begin{Theorem}
	\label{decomp:lemma:inv}
Assume that there exists a sGLSS of $(\yb,\ub,\p)$ and that the following holds:
\begin{enumerate}
\item  $\hat{\mathbf{S}}_d=(\{\hat{A}_i, \hat{B}_{i} \}_{i=1}^{\pdim}, \hat{C}, \hat{D},\ub)$ is an asGLSS of $(\yb^d,\p)$. 
\item  $\hat{\mathbf{S}}_s=(\{\hat{A}_{i}, \hat{K}_{i} \}_{i=1}^{\pdim}, \hat{C}, I_{\ny},\vb)$
is an asGLSS  of $(\yb^s,\p)$ in \emph{innovation form}, i.e. $\vb=\eb^s$. 
\end{enumerate}
Then $\hat{\mathbf{S}}=(\{\hat{A}_{i},\hat{K}_{i},\hat{B}_{i}\}_{i=1}^{\pdim},\hat{C},\hat{D},I,\eb^s)$ is a sGLSS of $(\yb,\ub,\p)$, and 
 the innovation process $\eb^s$ satisfies
 $\eb^s(t)=\y(t)-\hat{\y}(t)$, where
\begin{equation} 
\label{decomp:lemma:innov2}
\hat{\y}(t)=E_l[\yb(t) \mid \{\z^{\yb}_w (t), \z^{\ub}_w(t) \}_{w \in \Sigma^{+}} \cup \{\ub(t)\}].
\end{equation}
\end{Theorem}
The proof of Theorem \ref{decomp:lemma:inv} is presented in Section ~\ref{app:proof_lem2}. 
\begin{Remark}
\label{rem:diff}
The condition that the $\{A_i\}_{i=1}^{\pdim}$ and $C$ matrices of $\hat{\mathbf{S}}_d$ of 
$\hat{\mathbf{S}}_d$ can be relaxed: 
if $\bar{\mathbf{S}}_d=(\{\hat{A}_i^d, \hat{B}^d_{i} \}_{i=1}^{\pdim}, \hat{C}^d, \hat{D},\ub)$ and 
$\bar{\mathbf{S}}_s=(\{\hat{A}_i^s, \hat{B}^s_{i} \}_{i=1}^{\pdim}, \hat{C}^s, I,\eb^s)$ are
asGLSS of $(\y^s,\p)$ and $(\y^s,\p)$ respectively, then with
\[ \hat{A}_i\!\!=\!\!\begin{bmatrix} \hat{A}_i^d& 0 \\
0 & \hat{A}_i^s\end{bmatrix}, ~
 \hat{B}_i\!\!=\!\!\begin{bmatrix} \hat{B}_i^d \\  0 \end{bmatrix}, ~
  \hat{K}_i\!\!=\!\!\begin{bmatrix} 0 \\ \hat{K}_i^s \end{bmatrix}, ~
	\hat{C}=\begin{bmatrix} (C^d)^T \\ (C^s)^T \end{bmatrix}^T,
 \]
 $\hat{\mathbf{S}}_d$ and $\hat{\mathbf{S}}_s$
	from  Theorem \ref{decomp:lemma:inv}
	  are  asGLSSs of $(\y^d,\p)$ and $(\y^s,\p)$ respectively and Theorem \ref{decomp:lemma:inv} applies.
\end{Remark}
Thus, Theorem \ref{decomp:lemma} -- \ref{decomp:lemma:inv} means that finding sGLSSs of $(\y,\bu,\p)$ is equivalent to
finding an asGLSS representations of the deterministic and stochastic components respectively. 

Theorem \ref{decomp:lemma:inv}  suggests that $\eb^s(t)$ can be viewed as the innovation process of $\y$. Indeed, $\hat{\y}(t)$
from \eqref{decomp:lemma:innov2}
is the best linear prediction of $\y(t)$ based on 
past values of $\y$ and past and current values of $\bu$ 
multiplied by 
past values of the switching process. Then $\eb^s(t)$ is
the prediction error $\y(t)-\hat{\y}(t)$.  
This motivates the following definition.
\begin{Definition}[Innovation form]
The sGLSS \eqref{eqn:LPV_SSA} is
in \emph{innovation form}, if $F$ is the identity matrix and 
$\vb=\eb^s$. 
\end{Definition}
Similarly to the LTI case \cite{LindquistBook}, an sGLSS 
in innovation form can be viewed as a recursive 
filter driven by $\y$, $\bu$, $\p$, whose output is the optimal prediction 
$\hat{\y}(t)$ from \eqref{decomp:lemma:innov2}.
Indeed,  from $\eb^s(t)=\y(t)\!\!-\!\! C\x(t)\!\!-\!\! D\bu(t)$
it follows that $\x(t+1)$ is a function of $\x(t),\bu(t),\y(t)$ and $\p(t)$,  and $\hat{\y}(t)=C\x(t)+D\bu(t)$
\begin{Corollary}[Existence]
\label{col:ex}
 Any sGLSS of $(\y,\bu,\p)$ can be transformed to a sGLSS of $(\y,\bu,\p)$ in innovation form. 
\end{Corollary}
\begin{proof}
	From Theorem \ref{decomp:lemma} it follows that $\mathbf{S}_s$ is an asGLSS of $(\y^s,\p)$ and $\mathbf{S}_d$ is an 
	asGLSS of $(\y^d,\p)$. From \cite[Theorem 2]{PetreczkyBilinear} it follows that there exists a (minimal) asGLSS
	$\bar{\mathbf{S}}_s$ of $(\y^s,\p)$ in innovation form and by \cite[Theorem 3]{PetreczkyBilinear} it can be computed 
	from $\mathbf{S}_s$ using \cite[Algorithm 1]{PetreczkyBilinear}. Then using Remark \ref{rem:diff} and 
	Theorem \ref{decomp:lemma:inv}, it follows that $\hat{\mathbf{S}}$ defined in Theorem \ref{decomp:lemma:inv}
	is a sGLSS of $(\y,\bu,\p)$ in innovation form. 
\end{proof}
We can also provide conditions for minimality of sGLSSs. To this end, for a sGLSS of the form \eqref{eqn:LPV_SSA} we refer to the
dimension $\nx$ of the state-space as \emph{dimension} of sGLSS. 
\begin{Corollary}[Minimality]
\label{col:min}
	Assume that $\mathbf{S}$ is a sGLSS of $(\y,\bu,\p)$, 
	and assume that $\mathbf{S}_s$ from Theorem \ref{decomp:lemma} is observable and reachable in the terminology of \cite{PetreczkyBilinear}, if viewed as a stationary \GBS.  
	Then it is minimal dimensional among all the sGLSSs of $(\y,\bu,\p)$. 
\end{Corollary}
\begin{proof}
Let  $\hat{\mathbf{S}}$ be a sGLSS of $(\y,\bu,\p)$ of smaller dimension than $\mathbf{S}$. Then by Theorem \ref{decomp:lemma},
	$\hat{\mathbf{S}}_s$ is an asGLSS of $(\y,\p)$ of the same dimension as $\hat{\mathbf{S}}$. However, from  \cite[Theorem 2]{PetreczkyBilinear},
	$\mathbf{S}_s$ is a minimal dimensional asGLSS of $(\y,\p)$ and it is of the same dimension as $\mathbf{S}$, i.e. of dimension larger than $\hat{\mathbf{S}}_s$,
	which a contradiction. 
\end{proof}
Note that observability and reachability in the sense of \cite{PetreczkyBilinear} can be characterized by rank conditions of suitable
matrices, which can be constructed from the matrices of $\mathbf{S}_s$.  
We also get the following sufficient condition for isomorphism of minimal sGLSSs in innovation form. 
\begin{Corollary}[Isomorphism]
\label{col:min:is}
	Assume $\mathbf{S}$ is of the form \eqref{eqn:LPV_SSA} and $\hat{\mathbf{S}}=(\{\hat{A}_{\sigma},\hat{B}_{\sigma},\hat{K}_{\sigma}\}_{\sigma \in \Sigma},\hat{C},\hat{D},I,\eb^s)$ 
	and they are both sGLSS of $(\y,\bu,\p)$ in innovation form and $\mathbf{S}_s$ and $\hat{\mathbf{S}}_s$ are both reachable and observable 
	as stationary \GBS\ in the terminology of \cite{PetreczkyBilinear}. Assume that
	the covariance matrix $E[\eb^s(t)(\eb^s(t))^T\p_{\sigma}^2(t)]$ is nonsingular 
	and  $\mathrm{Im}[B_{\sigma}^T,\hat{B}_{\sigma}^T]^T \subseteq \mathrm{Im}[K_{\sigma}^T,\hat{K}_{\sigma}^T]^T$, 
 $\sigma \in \Sigma$ and $\hat{D}=D$. 
        Then there exists a non-singular  matrix $T$ such that 
        for all $\sigma \in \Sigma$, 
	\begin{equation}
    \label{iso:eq1}
       ~ TA_{\sigma}T^{-1}\!\!=\!\!\hat{A}_{\sigma}, ~ T[K_{\sigma}, B_{\sigma}]=[\hat{K}_{\sigma}, \hat{B}_{\sigma}], 
       ~ CT^{-1}\!\!=\!\! \hat{C}
    \end{equation}
\end{Corollary}
\begin{proof}
  Since both $\mathbf{S}_s$ and $\hat{\mathbf{S}}_s$  are both minimal asGLSS of $(\y,\p)$ in innovation form, and by
	\cite[Theorem 2]{PetreczkyBilinear}, they are isomorphic, i.e., there exists a non-singular  matrix $T$ such that  
   $TA_{\sigma}T^{-1}=\hat{A}_{\sigma}$, $TK_{\sigma}=\hat{K}_{\sigma}$, $CT^{-1}=\hat{C}$ holds. Since 
   $\mathrm{Im}[B_{\sigma}^T,\hat{B}_{\sigma}^T]^T \subseteq \mathrm{Im}[K_{\sigma}^T,\hat{K}_{\sigma}^T]^T$,
   for some matrix $Z_{\sigma}$, 
   $B_{\sigma}=K_{\sigma}Z_{\sigma}$, $\hat{B}_{\sigma}=\hat{K}_{\sigma}Z_{\sigma}$, from which
   \eqref{iso:eq1} follows.
\end{proof}
Corollary \ref{col:min:is} provides sufficient conditions for two sGLSSs in innovation form to be isomorphic, and hence have the same 
deterministic behavior, as defined in Section \ref{sect:intro}. 



\section{Conclusion}
We have shown that 
outputs of stochastic  generalized linear switched systems can be decomposed into two parts, deterministic and stochastic one, and we used it to derive existence
of representation in innovation form and to formulate sufficient conditions for minimality and uniqueness of such representations up to isomorphism. Future work will be directed towards extending these results for a larger class of inputs and switching signals.

\section{Appendix: Proofs of Theorems \ref{decomp:lemma} and \ref{decomp:lemma:inv}}
\subsection{Proof of Theorem \ref{decomp:lemma}}\label{appendix1}
The proof of is an extension
of \cite[proof of Lemma 1]{mejari2019realization}.
Let  $\mathcal{H}_{t,+}^{\ub}$ be the closed subspace of $\mathcal{H}_1$ (see Notation \ref{hilbert:notation}) generated by the components of  $\{ \z^{\ub}_{w}(t) \}_{w \in \Sigma^{+}} \cup \{\ub(t)\}$.
\begin{Lemma}
	\label{decomp:lemma:pf1}
	\label{decomp:lemma:pf2}
	The entries of the variables $\vb(t)$ and $\{\z_{w}^{\vb}(t)\}_{w \in \Sigma^{+}}$ are orthogonal to $\mathcal{H}_{t,+}^{\ub}$. 
\end{Lemma}
\begin{pf}
        Define
	$\mathbf{r}(t):=\begin{bmatrix} \vb^T(t) & \ub^T(t) \end{bmatrix}^T$. By the definition of a sGLSS, $\mathbf{r}$ is ZMWSII and a
        white noise process w.r.t. $\p$. Moreover, $\vb(t)$ is the upper $\nn$ block of $\r(t)$.
        Since $\mathbf{r}$ is a white noise w.r.t. $\p$, 
	$E[\mathbf{r}(t)(\z^{\mathbf{r}}_{w}(t))^T]=0$, and  $\frac{1}{\sqrt{p_{\sigma}}} E[\vb(t)(\z_{w}^{\ub}(t))^T]$ is the lower-left block of 
        that latter matrix, and hence it is also zero.
	Since  $E[\vb(t)(\ub(t))^T\p_{i}^2(t)]=0$ and  $E[\vb(t)(\ub(t))^T\p_{i}(t)\p_j(t)]=0$ for $i \ne j$ due to $\mathbf{r}$ being ZMWSII (\cite[Lemma 7]{PetreczkyBilinear}),
	and  	$\sum_{i=1}^{n_p} \alpha_i \p_i=1$ for some $\{\alpha_i\}_{i=1}^{\pdim}$, it follows
	$E[\vb(t)(\ub(t))^T]=\sum_{i,j=1}^{\pdim} \alpha_i\alpha_j E[\vb(t)(\ub(t))^T\p_i(t)\p_j(t)]=0$.
	That is, $\vb(t)$ is orthogonal to $\mathcal{H}_{t,+}^{\ub}$. 
	Since $\mathbf{r}(t)$ is a ZMWSII, from \cite[Lemma 7]{PetreczkyBilinear} it follows that
	$E[\z_{w}^{\mathbf{r}}(t)(\z^{\mathbf{r}}_{v}(t))^T]=0$ for all $v \in \Sigma^{+}$, $v \ne w$ or $v \notin L$ or $w \notin L$, and 
	if $v=w \in L$ and $\sigma$ is the first letter of $w$, then  $E[\z_{w}^{\mathbf{r}}(t)(\z^{\mathbf{r}}_{w}(t))^T]=E[\z_{\sigma}^{\mathbf{r}}(t)(\z^{\mathbf{r}}_{\sigma}(t))^T]$.
	
	Since $E[\z_{w}^{\vb}(t)(\z^{\ub}_{v}(t))^T]$ is the upper right block of $E[\z_{w}^{\mathbf{r}}(t)(\z^{\mathbf{r}}_{v}(t))^T]$, it follows that 
	$E[\z_{w}^{\vb}(t)(\z^{\ub}_{v}(t))^T]=0$ if $v \ne w$ and  $E[\z_{w}^{\vb}(t)(\z^{\ub}_{w}(t))^T]=E[\z_{\sigma}^{\vb}(t)(\z^{\ub}_{\sigma}(t))^T]=\frac{1}{p_{\sigma}} E[\ub(t-1)\vb(t-1)\p_{\sigma}^2(t-1)]$, where $\sigma$ is the first letter of $w$, and from  Definition \ref{def:Stationary}, it follows that the latter expectation is zero. That is, $E[\z_{w}^{\vb}(t)(\z^{\ub}_{v}(t))^T]=0$ for all $v \in \Sigma^{+}$. 
	
	Since $\mathbf{r}(t)$ is a 
	white noise w.r.t. $\p$,
    by \cite[Lemma 7]{PetreczkyBilinear}
    $E[\z_{w}^{\mathbf{r}}(t)(\mathbf{r}(t))^T]=E[\z_{ws}{\mathbf{r}}(t) (\z_{s}^{\mathbf{r}}(t))^T]=0$
    for any $s \in \Sigma^{+}$, and since  
	$E[\z_{w}^{\vb}(t)(\ub(t))^T]$ is the upper right block of $E[\z_{w}^{\mathbf{r}}(t)(\mathbf{r}(t))^T]$,  $E[\z_{w}^{\vb}(t)(\ub(t))^T]=0$. 
	Since $\z_w^{\vb}(t)$ is uncorrelated with random variable which generate $\mathcal{H}_{t,+}^{\ub}$, the statement of the lemma follows.
\end{pf}

Let us denote by $\mathcal{H}_{t}^{\ub}$, the closed subspace generated by the components of  $\{ \z^{\ub}_{w}(t) \}_{w \in \Sigma^{+}}$. It is clear that $\mathcal{H}_{t}^{\ub} \subseteq \mathcal{H}_{t,+}^{\ub}$.
Define	$\xb^d(t)=E_l[\xb(t)  \mid \{ \z^{\ub}_{w}(t) \}_{w \in \Sigma^{+}} \cup \{\ub(t)\}]$.
\begin{Lemma}
	\label{decomp:lemma:pf3}
	The entries of $\xb^d(t)$ belong to $\mathcal{H}_{t}^{\ub}$ and 
	\begin{equation}
		\label{decomp:lemma:pf3:eq2}
		\xb^d(t) = \sum_{w \in \Sigma^{*}, \sigma \in \Sigma, \sigma w \in L} \sqrt{p_{\sigma w}} A_w B_{\sigma}\z^{\ub}_{\sigma w}(t),
	\end{equation}
	where the convergence is in the mean square sense. 
\end{Lemma}
\begin{pf}
	It is clear from the definition that the components of $\xb^d(t)$ belong to $\mathcal{H}_{t,+}^{\ub}$. 
	From Lemma \ref{decomp:lemma:pf2} it follows that,  $E_l[\z^{\vb}_{\sigma w}(t) \mid H_{t,+}^{\ub}]=0$, and since the components of $\z^{\ub}_{\sigma w}(t)$ belong to $\mathcal{H}_{t,+}^{\ub}$, it follows that $E_l[\z^{\ub}_{\sigma w}(t) \mid H_{t,+}^{\ub}]=\z^{\ub}_{\sigma w}(t)$, 
	Since \eqref{stat:state:eq1} holds
    and the map $z \mapsto E_l[z \mid M]$ (where $z \in \mathcal{H}_1$) is a continuous linear operator for any closed subspace $M$, 
    it follows that $\x^d(t)$ will be the infinite sum of the elements 
    $\sqrt{p_{\sigma w}} A_w \left(K_{\sigma} E_l[\z^{\vb}_{\sigma w}(t) \mid H_{t,+}^{\ub}] + B_{\sigma}E_l[\z^{\ub}_{\sigma w}(t) \mid H_{t,+}^{\ub}]\right)$, i.e.,
    \eqref{decomp:lemma:pf3:eq2} holds.
	Since the components of $\z^{\ub}_{\sigma w}(t)$ belong to $\mathcal{H}_{t}^{\ub}$, the components of the right-hand side of 
	\eqref{decomp:lemma:pf3:eq2} belong to $\mathcal{H}_{t}^{\ub}$ and hence the components of $\xb^d(t)$ belong to $\mathcal{H}_{t}^{\ub}$.
	The convergence of the right-hand side  of \eqref{decomp:lemma:pf3:eq2} in the mean square sense follows from that of 
	\eqref{stat:state:eq1}.
\end{pf}
\begin{Lemma}
	\label{decomp:lemma:pf4}
	Define  $\xb^s(t)=\xb(t)-\xb^d(t)$.
	The entries of $\xb^s(t)$ belong to $\mathcal{H}_{t}^{\vb}$, they are orthogonal to $\mathcal{H}_{t,+}^{\ub}$ and 
	\begin{equation}
		\label{decomp:lemma:pf4:eq2}
		\xb^s(t) = \sum_{w \in \Sigma^{*}, \sigma \in \Sigma, \sigma w \in L} \sqrt{p_{\sigma w}} A_w K_{\sigma}\z^{\vb}_{\sigma w}(t),
	\end{equation}
	where the sum converges in the mean-square sense.
\end{Lemma}
\begin{pf}
	From \eqref{decomp:lemma:pf3:eq2}, $\xb^s(t)=\xb(t)-\xb^d(t)$
	and   \eqref{stat:state:eq1}, 
	it follows that \eqref{decomp:lemma:pf4:eq2} holds. By Lemma \ref{decomp:lemma:pf2}, $\{\z_{w}^{\vb}(t)\}_{w \in \Sigma^{+}}$ are orthogonal to $\mathcal{H}_{t,+}^{\ub}$, hence
	all the summands in the right-hand side of \eqref{decomp:lemma:pf4:eq2} are orthogonal to $\mathcal{H}_{t,+}^{\ub}$.
\end{pf}
\begin{pf}[Proof of Theorem \ref{decomp:lemma}]
 Since $\sum_{\sigma \in \Sigma} p_{\sigma} A_{\sigma} \otimes A_{\sigma}$ is stable and $\ub$ and $\vb$ are both white noise processes w.r.t. $\p$,      
  and from  \eqref{decomp:lemma:pf3:eq2}-\eqref{decomp:lemma:pf4:eq2} 
 and \cite[Lemma 3]{PetreczkyBilinear} it follows that $\xb^d$ is the unique state process of $\mathbf{S}_d$ and $\xb^s$ is the unique state process of $\mathbf{S}_s$. 
Notice that 
	\[
	\begin{split}
		&\yb^d(t)=E_l[\yb(t) \mid \mathcal{H}_{t,+}^{\ub}]=\\&=CE_l[\xb(t) \mid \mathcal{H}_{t,+}^{\ub}]+D E_l[\ub(t) \mid \mathcal{H}_{t,+}^{\ub}]+E_l[\vb(t) \mid \mathcal{H}_{t,+}^{\ub}].
	\end{split}
	\]
	By Lemma \ref{decomp:lemma:pf1}, $\vb(t)$ is orthogonal to 
	$\mathcal{H}_{t,+}^{\ub}$, $E_l[\vb(t) \mid \mathcal{H}_{t,+}^{\ub}]=0$ and as the components $\bu(t)$ belong to $\mathcal{H}_{t,+}^{\ub}$,
	$E_l[\ub(t) \mid \mathcal{H}_{t,+}^{\ub}]=\ub(t)$. Hence, 
	\( \yb^d(t)=C\xb^d(t)+D\bu(t) \) and \( \yb^s(t)=C\xb^s(t)+F\vb(t) \). 
That is,   $\mathbf{S}_d$ is an asGLSS of $(\yb^d,\p)$ and $\mathbf{S}_s$ is an asGLSS of $(\yb^s,\p)$  respectively.
\end{pf}
\subsection{Proof of Theorem \ref{decomp:lemma:inv}}\label{app:proof_lem2}
The proof of Theorem \ref{decomp:lemma:inv} is an adaptation
of \cite[proof of Lemma 2]{mejari2019realization}. 
Assume that $\mathbf{S}$ of the form \eqref{eqn:LPV_SSA}  is a sGLSS of $(\yb,\ub,\p)$. 
Denote by $\mathcal{H}_{t,+}^{\vb}$ and $\mathcal{H}_t^{\vb}$  the closed subspaces of $\mathcal{H}_1$ (see Notation \ref{hilbert:notation}) 
generated by the components $\{\z^{\vb}_w(t)\}_{w \in \Sigma^{+}} \cup \{\vb(t)\}$ and $\{\z^{\vb}_w(t)\}_{w \in \Sigma^{+}}$ respectively.
\begin{Lemma}\label{decomp:lemma:inv:pf2.1}
	The  entries of 
	$\{\z^{\yb^s}_v(t), \z^{\eb^s}_v(t)\}_{v \in \Sigma^{+}}$, $\yb^s(t),\eb^s(t)$
	belong to $\mathcal{H}_{t,+}^{\vb}$.
\end{Lemma}
\begin{pf}
    Recall from the proof of Theorem \ref{decomp:lemma} that 
	$\yb^s(t)=C\xb^s(t)+\vb(t)$.
	Then by \eqref{decomp:lemma:pf4:eq2},
     the components of $\yb^s(t)$ belong to $\mathcal{H}_{t,+}^{\vb}$. 
	Then by \cite[Lemma 11]{PetreczkyBilinear}, the components of 
	$\z^{\yb^s}_v(t)$ belong to $\mathcal{H}_{t,+}^{\vb}$ and hence, $\mathcal{H}_{t}^{\yb^s} \subseteq \mathcal{H}_{t,+}^{\vb}$.
	Since $\eb^s(t)=\yb^s(t)-E_l[\yb^s(t) \mid \mathcal{H}_{t}^{\yb^s}] $, this then implies 
	that the components of 
	$\eb^s(t)$ belong to $\mathcal{H}_{t,+}^{\vb}$. Since $\z_{v}^{\vb}(t)=\sum_{i=1}^{\pdim} \alpha_i \z_{vi}^{\vb}(t+1)$, $\vb(t)=\sum_{i=1}^{\pdim} \alpha_i \z_{i}^{\vb}(t+1)$, as $\sum_{i=1}^{\pdim} \alpha_i \p_i =1$, it follows that
	$\mathcal{H}_{t,+}^{\vb} \subseteq \mathcal{H}_{t+1}^{\vb}$ and from  \cite[Lemma 11]{PetreczkyBilinear} it follows that
	the components of $\z_v^{\eb^s}(t)$ belong to $\mathcal{H}_{t}^{\vb} \subseteq \mathcal{H}_{t,+}^{\vb}$. 
\end{pf}

\begin{Lemma}
\label{decomp:lemma:inv:pf2.2} 
The entries of $\{\z^{\yb^s}_v(t),\z^{\eb^s}_v(t)\}_{v \in \Sigma^{+}}$,
$\yb^s(t)$ and $\eb^s(t)$ are orthogonal to $\mathcal{H}_{t,+}^{\ub}$.
\end{Lemma}
\begin{pf}
	Using Lemma \ref{decomp:lemma:pf1} and  
	it follows that the elements of $\mathcal{H}_{t,+}^{\vb}$ are orthogonal to $\mathcal{H}_{t,+}^{\ub}$.
	Since the coordinates of  $\yb^s(t), \eb^s(t),  \{\z^{\yb^s}_v(t),\z^{\eb^s}_v(t)\}_{v \in \Sigma^{+}}$ 
	belong to $\mathcal{H}_{t,+}^{\vb}$, 
	the statement follows. 
\end{pf}
\begin{Lemma}\label{decomp:lemma:inv:pf2}
	$\r=\begin{bmatrix}  (\eb^s)^T & \ub^T \end{bmatrix}^T$ is a white noise process w.r.t. $\p$ and $E[\eb^s(t)\ub^T(t)\p_{\sigma}^2(t)]=0$ for all $\sigSet$.
\end{Lemma}
\begin{pf}
	We first show that $\r$ is a ZMWSII, by showing that $\r$  satisfies the conditions of  Definition \ref{def:ZMWSSI} one by one. 
	First, we show that the processes $\r(t),\z_w^{\r}(t),w \in \Sigma^{+}$ is zero mean, square integrable.
	
	By assumption  $\ub$ is a ZMWSII and white noise process w.r.t. $\p$.
	From the fact that $\hat{\mathbf{S}}_s$ is an asGLSS of $(\yb^s,\p)$ it follows that $\eb^s$ is also a ZMWSII. 
	 Thus $\eb^s(t),\ub(t), \{\z_w^{\eb^s}(t),\z_{w}^{\ub}(t)\}_{w \in \Sigma^{+}}$ is zero mean, square integrable.
	 From this it follows that  $\r(t)$ and
	$\z_w^{\r}(t)$ 
are zero mean and square integrable.

	From Lemma \ref{decomp:lemma:inv:pf2.1} it follows that the components $\eb^s(t)$ belongs to $\mathcal{H}_{t,+}^{\vb}(t)$. 
	Moreover, by definition of sGLSS, $\mathbf{w}=\begin{bmatrix} \vb^T & \ub^T \end{bmatrix}^T$ is ZMWSII. 
	Hence, with the notation of 
	Definition \ref{def:ZMWSSI}, the $\sigma$-algebras $\mathcal{F}_t^{\mathbf{w}}$ and $\mathcal{F}_t^{\p,+}$ are conditionally independent  
      w.r.t. $\mathcal{F}_t^{\p,-}$. 
	From the fact that $\eb^s(t)$ belongs to $\mathcal{H}_{t,+}^{\vb}$ it follows that $\eb^s(t)$ is measurable with respect to the $\sigma$-algebra generated by
	$\{\vb(t)\} \cup \{ \z_v^{\vb}(t)\}_{v \in \Sigma^{+}}$ and the latter $\sigma$-algebra is a subset of $\mathcal{F}^{\mathbf{w}}_t \lor \mathcal{F}^{\p,-}_t$, 
	where for two $\sigma$-algebras $\mathcal{F}_i$, $i=1,2$, $\mathcal{F}_1 \lor \mathcal{F}_2$ denotes the smallest $\sigma$-algebra generated by the $\sigma$-algebras $\mathcal{F}_1,\mathcal{F}_2$. 
	That is, $\eb^s(t)$ is measurable w.r.t. the $\sigma$ algebra $\mathcal{F}^{\mathbf{w}}_t \lor \mathcal{F}^{\p,-}_t$. Hence,
	$\mathcal{F}_t^{\r} \subseteq \mathcal{F}^{\mathbf{w}}_t \lor \mathcal{F}^{\p,-}_t$. 
Since  $\mathcal{F}^{\mathbf{w}}_t$ and $\mathcal{F}^{\p,+}$ are conditionally independent w.r.t.  $\mathcal{F}^{\p,-}_t$, from \cite[Proposition 2.4]{vanputten1985} it follows that
	$\mathcal{F}^{\mathbf{w}}_t \lor \mathcal{F}^{\p,-}_t$ and $\mathcal{F}^{\p,+}_t$ are conditionally independent w.r.t. 
	$\mathcal{F}^{\p,-}_t$, and as $\mathcal{F}^{\r}_t \subseteq \mathcal{F}^{\mathbf{w}}_t \lor \mathcal{F}^{\p,-}_t$, 
	it follows that   $\mathcal{F}^{\r}_t$ and $\mathcal{F}^{\p,+}$ are conditionally independent w.r.t. $\mathcal{F}^{\p,-}_t$.

	Finally, we show that $\r(t),\z_w^{\r}(t),w \in \Sigma^{+}$
	are   jointly wide-sense stationary, i.e., for all $s,t,\in \mathbb{Z}$, $s \le t$, $v,w \in \Sigma^{+}$, 
    $E[\mathbf{h}_1(t)(\mathbf{h}_1(t)^T]$, where
	$\mathbf{h}_1(t),\mathbf{h}_2(t) \in \{\r(t)\} \cup \{\z_w^{\r}(t),w \in \Sigma^{+}\}$ does not depend on $t$.  
	We show only the case,
 $\expect{\r(t)(\zwr(s))^{T}} = \expect{\r(t-s) (\zwr(0))^{T}}$, the proof of the general case is similar.
	From Lemma \eqref{decomp:lemma:inv:pf2}  it follows that $\expect{\z^{\eb^s}_w(t+k)(\z_{v}^{\ub}(s+k))^T}=0$, and hence
    the matrix
    $\expect{\zwr(t)(\zvr(s))^{T}}$ is a block-diagonal one, where the blocks on the diagonal
    are $\expect{\z^{\eb^s}_w(t)(\z_{v}^{\eb^s}(s))^T}$ and $\expect{\z^{\ub}_w(t)(\z_{v}^{\ub}(s))^T}$ and by $\ub$ and $\e^s$ being
    ZMWSII, it follows that the latter do not
    depend on $t$.
	That is, we have shown that $\r$ satisfies all the conditions of  Definition \ref{def:ZMWSSI}.

	Next we show that $\r$ is a white noise process w.r.t. $\p$. To this end, by \cite[Lemma 7]{PetreczkyBilinear}, it is enough to
	show that $E[\r(t)(\z^{\r}_w(t))^T]=0$ for all $w \in \Sigma^{+}$. From 
	Lemma \ref{decomp:lemma:inv:pf2}
    it follows that $\expect{\r(t) (\zwr(t))^{T}}$
    is block diagonal, with the block on the diagonal
    being $\expect{\eb^s(t)(\z_{w}^{\eb^s}(t))^T}$, $\expect{\ub(t)(\z_{w}^{\ub}(t))^T}$,
    and the latter are zero as $\eb_s$ and $\ub$
    are white noise w.r.t $\p$. 
	Finally, $E[\eb^s(t)\ub^T(t)\p_{\sigma}^2(t)]=0$ for all $\sigSet$ follows from  Lemma \ref{decomp:lemma:inv:pf2.2}.
\end{pf}
\begin{pf}[Proof of Theorem \ref{decomp:lemma:inv}]
	From Lemma \ref{decomp:lemma:inv:pf2} it follows that the noise process $\eb^s$ and the input $\ub$ satisfy the condition of $E[\eb^s(t)(\ub(t))^T\p_{\sigma}^2(t)]=0$, $\sigma \in \Sigma$.
 Since $\hat{\mathbf{S}}_s$ and $\hat{\mathbf{S}}_d$ are both asGLSS, it follows that $\sum_{i=1}^{\pdim} p_{i} \hat{A}_i \otimes \hat{A}_i$ is stable.
Hence $\hat{\mathbf{S}}$ satisfies the conditions of a sGLSS.
 Let $\hat{\xb}^s$ and $\hat{\xb}^d$ be the unique state processes of $\hat{\mathbf{S}}_s$ and $\hat{\mathbf{S}}_d$ respectively. We claim that
	$\hat{\xb}(t)=\hat{\xb}^d(t)+\hat{\xb}^s(t)$ is the unique state process of $\hat{\mathbf{S}}$. 
	Indeed, $\hat{\xb}(t+1)=\sum_{i=1}^{\pdim} (\hat{A}_i\hat{\xb}()+\hat{B}_i\bu(t)+\hat{K}_i\eb^s(t))\pi_1(t)$ holds and $\hat{\xb}(t)$ is a ZMWSII, as it is
	a sum of two ZMWSII processes. Hence, $\hat{\xb}(t)$ is the uniqe state process of $\hat{\mathbf{S}}$.
Finally, 
from $\yb^d(t)=\hat{C}\hat{\xb}^d(t)+\hat{D}\ub(t)$ (as $\hat{\mathbf{S}}_d$ is an asGLSS of $(\yb^d,\p)$) and $\yb^s(t)=\hat{C}\hat{\xb}^s(t)+\eb^s(t)$ (as $\hat{\mathcal{S}}_s$ is an asGLSS of $(\yb^s,\p)$), it follows
that $\yb(t)=\hat{C}\hat{\xb}(t)+\hat{D}\ub(t)+\eb^s(t)$, i.e., $\hat{\mathbf{S}}$ is a sGLSS of $(\yb,\ub,\p)$. 
\end{pf}


\bibliographystyle{plain}
\bibliography{Bib.bib}

\end{document}